\documentclass[12pt,reqno]{article}

\usepackage[usenames]{color}
\usepackage{amssymb}
\usepackage{graphicx}
\usepackage{amscd}

\usepackage[colorlinks=true,
linkcolor=webgreen,
filecolor=webbrown,
citecolor=webgreen]{hyperref}

\definecolor{webgreen}{rgb}{0,.5,0}
\definecolor{webbrown}{rgb}{.6,0,0}

\usepackage{color}
\usepackage{fullpage}
\usepackage{float}

\usepackage{graphics,amsmath,amssymb}
\usepackage{amsthm}
\usepackage{amsfonts}
\usepackage{latexsym}
\usepackage{epsf}

\usepackage{verbatim}
\usepackage{enumerate}
\usepackage{mathtools}

\usepackage{tikz}

\usepackage[blocks]{authblk}

\usepackage{dcolumn}

\theoremstyle{plain}
\newtheorem{theorem}{Theorem}
\newtheorem{corollary}[theorem]{Corollary}
\newtheorem{lemma}[theorem]{Lemma}
\newtheorem{proposition}[theorem]{Proposition}

\theoremstyle{definition}

\newtheorem{example}{Example}

\theoremstyle{remark}

\newtheorem*{remark*}{Remark}

\begin{document}

\title{Chip-Firing on Infinite $k$-ary Trees}

\author{Dillan Agrawal}
\author{Selena Ge}
\author{Jate Greene}
\author{Dohun Kim}
\author{Rajarshi Mandal}
\author{Tanish Parida}
\author{Anirudh Pulugurtha}
\author{Gordon Redwine}
\author{Soham Samanta}
\author{Albert Xu}
\affil{PRIMES STEP}
\author{Tanya Khovanova}
\affil{MIT}
%\date

\maketitle

\begin{abstract}
We use an infinite $k$-ary tree with a self-loop at the root as our underlying graph. We consider a chip-firing process starting with $N$ chips at the root. We describe the stable configurations. We calculate the number of fires for each vertex and the total number of fires. We study a sequence of the number of root fires for a given $k$ as a function of $N$ and study its properties. We do the same for the total number of fires.
\end{abstract}

\section{Introduction}

\subsection{History and Background}

In 1975, Arthur Engel \cite{Eng75} introduced the probabilistic abacus, an early version of chip-firing. In 1986, Joel Spencer \cite{Spe86} independently discovered one-dimensional chip-firing. In 1991, Björner, Lovász, and Shor introduced a chip-firing game on graphs.

The chip-firing game is played on a graph and is an important part of structural combinatorics. In this game, each vertex has a certain number of chips and can fire if the number of chips is greater than or equal to its degree. When fired, a vertex sends a single chip to each neighbor. The game ends when we can no longer fire any vertices. The resulting configuration is called a stable configuration.

Chip-firing is similar to Abelian sandpiles, which were introduced in 1987 by Bak, Tang, and Wiesenfeld \cite{BTW87}. In Abelian sandpiles, each place on a finite grid has a slope. Each slope eventually increases as grains of sand are randomly placed on the grid. When the slope reaches a certain height, they collapse and split the sand grains evenly among their neighbors on the grid. Abelian sandpiles were the first dynamical system discovered to have self-organized criticality. 

Chip-firing has deep connections to the Tutte polynomial and group theory. More details can be found in Merino's survey from 2005 \cite{Mer05}.

More recently, Musiker and Nguyen \cite{MN23} have discovered many of the properties of chip-firing on infinite binary trees when starting with $2^n-1$ chips at the root in both labeled and unlabeled chip-firing. Inagaki, Khovanova, and Luo \cite{IKL24} have also examined infinite binary trees. They found formulae for the total number of times each vertex fires, the total number of times the root fires, and the total number of fires. 

Continuing the work of Inagaki, Khovanova, and Luo, we generalize some of the formulae they found to infinite $k$-ary trees. We calculate the number of times each vertex fires, the number of times the tree's root fires, and the total number of fires in a $k$-ary tree. We also calculate and study many old and new integer sequences.

In Section~\ref{sec:prelim}, we provide preliminary definitions and results that are important for the rest of the paper. The graph we use is a rooted, undirected infinite $k$-ary tree with a self-loop at the root, and our chip-firing game starts with $N$ chips at the root.

In Section~\ref{sec:stableconfs}, we show that if we start with $N$ chips at the root of our tree, the stable configuration will form a finite tree of height $n$, where $n$ is the integer such that $\frac{k^{n} - 1}{k - 1} \leq N \leq \frac{k^{n + 1} - k}{k - 1}$. The number of chips on a vertex in the stable configuration is dependent only on its layer and is a function of the $k$-ary expansion of $N - \frac{k^{n}-1}{k-1}$.

In Section~\ref{sec:vertexfires}, we calculate the number of times each vertex fires, given the total number of chips ($N$). In Section~\ref{sec:rootfires}, we use the formula from the previous section to explore the number of root fires in more detail. Given $k$, the number of root fires as a function of $N$---the total number of chips---has an interesting structure. Its first and second difference sequences exhibit fractal properties, which we study in more detail.

In Section~\ref{sec:totalfires}, we calculate the total number of fires and the corresponding difference sequences. We consider a case for $N = \frac{k^n-1}{k-1}$, which plays a special role in our paper. We explain the connection between our sequences and schizophrenic numbers, which are numbers combining properties of rational and irrational numbers.

\section{Preliminaries}
\label{sec:prelim}

This paper studies a chip-firing game on a specific graph: an infinite undirected rooted $k$-ary tree with a self-loop at the root.

\subsection{The underlying graph}

In a \textit{rooted tree}, we distinguish one vertex as the \textit{root}. Every other vertex $v$ has exactly one \textit{parent} vertex $v_p$, defined as the first vertex after $v$ in the shortest path from $v$ to the root. The vertex $v$ is said to be a \textit{child} of the vertex $v_p$. Any vertex other than vertex $v$ on the shortest path from $v$ to root is an \textit{ancestor} of $v$, while $v$ is a \textit{descendant} of each of its ancestors.

\begin{figure}[ht!]
\begin{center}
\includegraphics[scale=0.15]{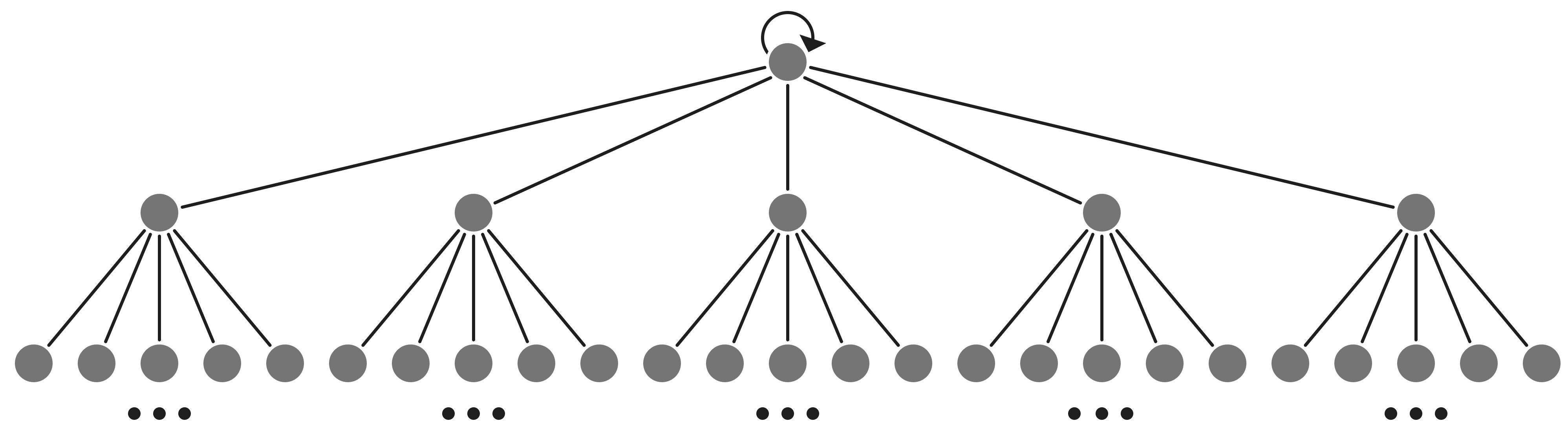}
\end{center}
\caption{An infinite undirected rooted 5-ary tree with a self-loop at the root}
\label{fig:5tree}
\end{figure}

We define a vertex $v$ as being on \textit{layer} $i+1$ if the shortest path from the root to $v$ traverses $i$ edges. By this definition, the root $r$ is on layer 1.

Sometimes, we need to refer to a subgraph of our underlying graph. For this, we need some more definitions. 

The \textit{height} of a tree is the length of the longest path from the root to a leaf. In a \textit{perfect} $k$-ary tree, every non-leaf vertex has $k$ children, and all leaves are the same distance from the root. In other words, all leaves are located on the same layer $h+1$, where $h$ is the height of such a tree.

\subsection{Chip-firing process}

We consider a standard unlabeled chip-firing game \cite{Kli19}, where some number of indistinguishable chips are placed at each vertex. A vertex can \textit{fire} if it has at least the same number of chips as its degree. In our case, a vertex can \textit{fire} when it has at least $k+1$ chips. When a vertex \textit{fires}, it sends one chip along each edge. We consider the chip-firing game where the \textit{initial state} has $N$ chips at the root and no chips anywhere else. The vertices fire until no firing is possible.

A \textit{stable configuration} is the distribution of chips on the graph where no vertex can fire; in other words, in a stable configuration, every vertex has fewer chips than its degree. In our case, we reach a stable configuration when every vertex has at most $k$ chips.

The following theorem, proved in \cite{BLS91}, describes when the chip-firing game may stabilize. It applies to any connected undirected graph $G=(V,E)$, where $V$ is the set of vertices and $E$ is the set of edges. We denote the number of vertices and edges by $|V|$ and $|E|$, respectively.

\begin{theorem}[Theorem 2.3 in \cite{BLS91}]
Let $N$ be the total number of chips. We have:
\begin{itemize}
    \item If $N > 2|E| - |V|$, then the game is infinite and never reaches a stable configuration.
    \item If $N < |E|$, then the game is finite, reaching a stable configuration.
    \item If $|E| \leq N \leq 2|E| - |V|$, then there exists an initial chip configuration that leads to an infinite game and a configuration that leads to a finite game.
\end{itemize}
\end{theorem}

In our case, we have an infinite $k$-ary tree with a finite number of chips; therefore, we have $N < |E|$, implying that the configuration leads to a finite game.

Klivans \cite{Kli19} studied many properties of unlabeled chip-firing. In particular, we use the following property from this work called the \textit{Global Confluence}.

\begin{theorem}[Theorem 2.2.2 in \cite{Kli19}]
If we can reach a stable configuration with a finite number of fires, then the stable configuration is unique. 
\end{theorem}

This theorem implies that the order of fires does not matter. In our case, the game is finite, implying that we always reach a unique stable configuration.

\subsection{Base $k$}

Our formulae related to $k$-ary trees can be better expressed in base $k$. Hence, we use the notation $(m)_k$ or $m_k$ to indicate that a string of digits $m$ is written in base $k$. For example, $111_2 = (111)_2 = 7$. Correspondingly, we will write $[n]_k$ to denote the string of digits representing a decimal integer $n$ in base $k$. Thus, $[n]_k = m$ is equivalent to $n = (m)_k$.

We denote by $\nu_k(n)$ the largest power of $k$ that divides $n$. In other words, $\nu_k(n)$ is the number of trailing zeros in $[n]_k$. When $k$ is prime, the value $\nu_k(n)$ is called the $k$-adic valuation of $n$. We will extend the terminology to composite numbers as well. In other words,
\[k^{\nu_k(n)} | n \quad \textrm{ and } \quad k^{\nu_k(n)+1} \nmid n.\]

\section{Stable configurations}
\label{sec:stableconfs}

In Proposition 3.3 from \cite{MN23}, the stable configurations are described for the case $k=2$. We describe the stable configurations for any $k \ge 2$.

\begin{proposition}
If we start with $N$ chips at the root, where $\frac{k^{n} - 1}{k - 1} \leq N \leq \frac{k^{n + 1} - k}{k - 1}$, then the vertices containing chips in the stable configuration form a perfect $k$-ary tree with height $n-1$. Furthermore, every vertex on the same layer has the same number of chips.
\end{proposition}

\begin{proof}
We prove this proposition by induction on $n$. When $n = 1$, the number of chips at the root ranges from $1$ to $k$, which is already a stable configuration, so this is a perfect $k$-ary tree with a height of 0. As only the root has chips in a stable configuration, every vertex on the same layer has the same number of chips.

Suppose the proposition holds for $n = i$; we will prove that it holds for $n = i + 1$. Suppose we start with $N$ chips at the root, where $\frac{k^{i + 1} - 1}{k - 1} \le N \le \frac{k^{i + 2} - k}{k - 1}$. We deploy the following firing strategy. 
\begin{enumerate}
\item Fire the root repeatedly until it cannot fire anymore.
\item Fire the root's children and their subtrees in parallel.
\item Whenever the $k$ children of the root fire, fire the root.
\item Repeat steps 2 and 3 until the stable configuration is reached.
\end{enumerate}

As this process is parallel, the root's children contain the same number of chips at any step. Every time the root fires, it sends 1 chip to itself. Therefore, the root always possesses at least 1 chip. After step 1, each child of the root contains between $\frac{k^{i} - 1}{k - 1}$ and $\frac{k^{i + 1} - k}{k - 1}$ chips. In step 3, when the $k$ children of the root fire, they each send 1 chip to the root. After that, the root gets $k$ chips, allowing it to fire, causing the root to return 1 chip to each child immediately. After step 1, we can interpret each subtree rooted at layer 2 as an independent tree with a self-loop at its root.

Therefore, the firing process of each subtree rooted at layer 2 is isomorphic to chip-firing on a rooted tree where the number of chips initially placed at the root is between $\frac{k^{i} - 1}{k - 1}$ and $\frac{k^{i + 1} - k}{k - 1}$, inclusive. By the induction hypothesis, once in the stable configuration, each subtree has the same number of chips on every vertex of the same layer. Moreover, vertices with chips form a perfect $k$-ary tree with height $n = i-1$. As the firing processes in all $k$ root's subtrees are the same, the stable configurations of the subtrees are the same. Thus, after the process stabilizes, the vertices containing chips in the whole tree form a perfect $k$-ary tree with height $i$, and every vertex on the same layer has the same number of chips. Since we are considering unlabeled chip-firing, the global confluence property applies, so the stable configuration is always as described above, finishing the proof.
\end{proof}

\begin{proposition}
\label{prop:a2c}
If we start with $N$ chips at the root, where $\frac{k^n-1}{k-1} \le N \le \frac{k^{n+1}-k}{k-1}$, then for $0\le i \le n-1$, the resulting stable configuration has $a_i+1$ chips on each vertex on layer $i+1$, where $a_{n-1}\dots a_2a_1a_0$ is the base $k$ expansion of $N - \frac{k^{n}-1}{k-1}$ with possible leading zeros.
\end{proposition}

\begin{proof}
We know that all vertices with chips form a perfect $k$-ary tree, where each vertex with chips has from 1 to $k$ chips, inclusive, and each vertex in the same layer has the same number of chips. Suppose each vertex on layer $i+1$ has $j_{i}$ chips. As there are $k^{i}$ vertices on layer $i+1$, the total number of chips is $\sum_{i=0}^{n-1} j_{i} k^{i}$, which equals $N$. Thus, the distribution of chips into layers in the final configuration implies representing $N$ in base $k$ with digits between $1$ and $k$. Since this representation is unique, to prove that the above formula describes the correct distribution of chips, we need to demonstrate that the total number of chips equals $N$. When we place $a_i+1$ chips on each vertex, we are placing a total of
\begin{align*}
 \sum_{i=0}^{n-1} k^{i}(a_i+1) &= \sum_{i=0}^{n-1} k^{i}a_i + \sum_{i=0}^{n-1} k^{i}\\
 &= N - \frac{k^n-1}{k-1} + \frac{k^n-1}{k-1}\\
 &= N
\end{align*}
chips, as desired.
\end{proof}

We see that numbers of the form $N = \frac{k^{n} - 1}{k - 1}$ play a special role, as the stable configuration reaches layer $n$ for the first time when we start with the number of chips in this form placed at the root.

From now on, let $a_{n - 1} \ldots a_{2}a_{1}a_{0}$ be the base $k$ expansion of $N - \frac{k^{n}-1}{k-1}$ with possible leading zeros. When we need to emphasize the dependency of these digits $a_i$ on the base $K$ and the number $N$, we will use the notation $a_i(N,k)$ instead of $a_i$. Similar to \cite{MN23, IKL24}, for the stable configuration starting with $N$ chips at the root in a $k$-ary tree, let us denote the number of chips at each vertex on layer $i+1$ as $c_{i}(N,k)$ for $0 \leq i \leq n-1$. Proposition~\ref{prop:a2c} states that $c_{i}(N,k) = a_{i}(N,k) + 1$. When the values of $N$ and $k$ are fixed, we will use the notations $a_i$ and $c_i$ to simplify the calculations.

The number $n$ also plays a special role in our paper. It was defined through inequalities
\[\frac{k^n-1}{k-1} \le N \le \frac{k^{n+1}-k}{k-1}.\]
After multiplying every term by $k-1$ and adding 1, we get
\[k^n \le N(k-1)+1 \le k^{n+1}-k+1.\]
Taking the log base $k$ of both sides gives us
\[n \le \log_k(N(k-1)+1) \le \log_k(k^{n+1}-k+1) < n+1.\]
Thus, we can define $n$ as 
\[n = \lfloor \log_{k}(N(k-1) + 1) \rfloor.\]

\section{The number of times each vertex fires}
\label{sec:vertexfires}

In \cite{MN23}, Musiker and Nguyen counted the number of times each vertex fires for a binary tree when starting with $N = 2^\ell - 1$ unlabeled chips at the root. In \cite{IKL24}, this was extended to any number of chips at the root; however, it was limited to binary trees. Here, we generalize their results for a $k$-ary tree.

\subsection{The number of fires for a vertex on layer $i+1$}

Let $f_i(N,k)$ denote the total number of fires of each vertex on layer $i+1$. When $N$ and $k$ are fixed, we will sometimes use the notation $f_i$ to simplify the calculations. We first look at the difference in the number of fires between layers $i-1$ and $i$.

\begin{lemma}
\label{lem:diffFire}
The difference in the number of fires between vertices on layers $i+1$ and $i+2$ equals the number of chips in the stable configuration belonging to the descendants of one particular vertex on layer $i + 1$ divided by $k$. In other
words,
\[f_i(N,k) - f_{i+1}(N,k) = \sum\limits_{j=i+1}^{n-1}k^{j-i-1}c_j(N,k),\]
where $n = \lfloor \log_{k}(N(k-1) + 1)\rfloor$.
\end{lemma}

\begin{proof}
The number of times a vertex fires equals the total number of chips sent to all children of the vertex divided by $k$. For a given vertex, the number of chips sent to its children is equal to the number of chips placed at the vertex's descendants plus the number of chips sent back from its children. The number of chips sent back to the vertex from its children is $k$ times the number of fires of each child. Thus, the difference between the number of fires of a vertex $v$ and its child is the number of chips belonging to the descendants of $v$ in the stable configuration divided by $k$.

Let vertex $v$ be on layer $i+1$, and let $j > i$. The number of vertices on layer $j+1$ that are the descendants of $v$ is $k^{j-i}$. Each descendant contains $c_j(N,k)$ chips by Proposition~\ref{prop:a2c}. Adding them up across all of the vertices, the number of chips placed at the descendants of $v$ is $\sum\limits_{j=i+1}^{n-1}k^{j-i}c_j(N,k)$, and the result follows after dividing this by $k$.
\end{proof}

\begin{example}
As the chips on the last layer, $n$, never fire, we have $f_{n-1} = 0$. From Lemma~\ref{lem:diffFire}, we get $f_{n-2} = c_{n-1}$.
\end{example}

Now, we can calculate the number of fires for each vertex on layer $i+1$.

\begin{theorem}
\label{thm:fires}
Given the total number of chips $N$ and the index $n = \lfloor\log_k(N(k-1)+1) \rfloor$, the number of fires for each vertex on layer $i+1$ is:
\[f_i(N,k) = \sum_{j = 1}^{n-i - 1} \left(\frac{k^{j} - 1}{k - 1} \right) c_{i+ j}(N,k).\]
\end{theorem}

\begin{proof}
We prove this formula using backward induction, starting with $i = n-1$. As the last layer never fires, we have $f_{n - 1} = 0$. Our formula gives us the same: $f_{n - 1} = \sum_{j = 1}^{0} (\frac{k^{j} - 1}{k - 1} ) c_{i + j} = 0$. Thus, the base of the induction is established.

Suppose the formula is true for a given $i+1$. We want to calculate the number of fires for $i$. In other words, we want to find $f_{i}$ assuming that $f_{i+1} = \sum_{j = 1}^{n-i - 2} ( \frac{k^{j} - 1}{k - 1} ) c_{i + j +1}$.

From Lemma~\ref{lem:diffFire}, we have that
\[f_{i} - f_{i+1} = \sum_{j = i+1}^{n - 1} k^{j - i - 1} c_{j}.\]
Thus, after reordering terms and keeping track of indices, we get
\begin{align*}
    f_{i} & = f_{i+1} + \sum_{j = i+1}^{n - 1} k^{j - i - 1} c_{j} = \sum_{j = 1}^{n-i - 2} \left(\frac{k^{j} - 1}{k - 1} \right) c_{i + j+1} + \sum_{j = i+1}^{n - 1} k^{j - i - 1} c_{j}\\
    & = \sum_{j = 1}^{n-i - 2} \left(\frac{k^{j} - 1}{k - 1} \right) c_{i + j + 1} + \sum_{j = 0}^{n-i-2} k^{j}c_{i + j+1} = c_{i+1} + \sum_{j = 1}^{n-i - 2} \left( \frac{k^{j + 1} - 1}{k - 1} \right) c_{i+j+1}\\
    & = \sum_{j = 1}^{n-i - 1} \left(\frac{k^{j} - 1}{k - 1} \right) c_{i+ j}(N,k),
\end{align*}
concluding the proof.
\end{proof}

The following corollary allows us to express the number of fires of each vertex on layer $i+1$ through the number of fires of the root.

\begin{corollary}
We have
\[f_i(N, k) = f_0\left(\left\lfloor \frac{N - \frac{k^n - 1}{k - 1}}{k^i} \right\rfloor + \frac{k^{n-i}-1}{k-1}, k\right),\]
where $n=\lfloor\log_k(N(k-1)+1)\rfloor$.
\end{corollary}

\begin{proof}
The formula for the number of fires of each vertex on layer $i+1$ is
\[f_i(N,k) = \sum_{j=1}^{n-i-1} \left(\frac{k^j - 1}{k-1}\right) c_{i+j}(N, k).\]
This matches up with the total number of fires for the root of a tree if we start with $N'$ chips, where $[N']_k$ is the first $n-i$ digits of $[N]_k$. In other words, 
\[c_j(N',k) = c_{i+j}(N,k).\]
In this case,
\[f_0(N',k) = \sum_{j=1}^{n-i-1} \left(\frac{k^j - 1}{k - 1} \right) c_{j}(N',k) = f_i(N,k).\]

From Proposition~\ref{prop:a2c}, we know that $a_j(N',k) = a_{i+j}(N,k)$. This implies that
\[N' - \frac{k^{n-i}-1}{k-1} = \left\lfloor \frac{N - \frac{k^n - 1}{k - 1}}{k^i} \right\rfloor,\]
and the corollary follows.
\end{proof}

The corollary above implies that we can focus only on the number of root fires to understand the properties of the number of vertex fires on other layers.

\subsection{The number of root fires}
\label{sec:rootfires}

As a particular case of Theorem~\ref{thm:fires}, we get the number of root fires.

\begin{corollary}
\label{cor:rootfires}
Given the total number of chips $N$ and the index $n = \lfloor \log_{k}(N(k-1) + 1) \rfloor$, the number of times the root fires is
\[f_0(N,k) = \frac{1}{k-1}\sum_{j=1}^{n-1} (k^j-1)c_j(N,k).\]
\end{corollary}

\begin{example}
Suppose we have a ternary tree $(k=3)$, and the total number of chips is 9 $(N=9)$. This means that we have index $n = \lfloor \log_3(9\cdot 2+1) \rfloor =2$. The ternary representation of $N-\frac{k^2-1}{k-1}= 9-4 = 5$ is $12_3$, implying that $a_1 =1$ and $c_1=2$. This means that the total number of times the root fires is $\sum_{j=1}^{1} \left( \frac{3^j-1}{2} \cdot 2 \right)=2$ times.
\end{example}

As in \cite{IKL24}, we are looking for a recursive formula.

\begin{corollary}
\label{cor:f0recursion}
We have
\[f_0(N) = \left\lceil \frac{N}{k} \right\rceil - 1 + f_0\left(\left\lceil \frac{N}{k} \right\rceil - 1\right).\]
\end{corollary}

\begin{proof}
The number of times the root fires is equal to the number of times it fires initially plus the number of times its children fire.
\end{proof}

\subsection{Special case of $N = \frac {k^n-1}{k-1}$}

The case of $N = \frac{k^{n} - 1}{k - 1}$ is important throughout our paper, as the stable configuration reaches layer $n$ for the first time when we start with $\frac{k^{n} - 1}{k - 1}$ chips in this form placed at the root. Moreover, if we start with $\frac{k^n - 1}{k-1}$ chips at the root, then $c_i = 1$ for $0 \le i < n$. Thus, the number of fires for a vertex on level $i+1$ is
\[f_{i}(N, k) = \sum_{j=1}^{n-i-1} \left( \frac{k^j-1}{k-1} \right) = \frac{k^{n-i}-k(n-i)+(n-i)-1}{(k-1)^2}.\]
Plugging in $k=2$, we get
\[f_{i}(N, 2) = 2^{n-i} - 2(n-i) + (n-i) - 1 = 2^{n-i} - n + i - 1,\]
matching the answer from \cite{MN23}.

The total number of root fires when $N = \frac{k^{n} - 1}{k - 1}$ is
\[f_{0}(N, k) = \frac{1}{k - 1}\sum_{j = 1}^{n - 1}(k^{j} - 1) = \frac{\frac{k^{n} - 1}{k - 1} - n}{k - 1} = \frac{k^{n} - nk + (n - 1)}{(k - 1)^2}.\]
Plugging in $k=2$, we see that for $2^n-1$ chips on a binary tree, the root fires $2^n-n-1$ times, which is the same result as in \cite{MN23}. These numbers form sequence A000295 in the OEIS \cite{OEIS}. Starting from index 1, the first few terms are
\[0,\ 1,\ 4,\ 11,\ 26,\ 57,\ 120,\ 247,\ 502,\ 1013,\ \ldots.\]

When we plug in $k=3$, we get that there are $\frac{3^n-2n-1}{4}$ root fires, which corresponds to sequence A000340 in the OEIS with a shifted index. Similarly, when we plug in $k=4$ and $k=5$, we get $\frac{4^n-3n-1}{9}$ and $\frac{5^n-4n-1}{16}$ root fires, which correspond to sequences A014825 and A014827 in the OEIS \cite{OEIS}. These sequences appear later in our paper in Table~\ref{tab:a}.

\subsection{Differences}

Notice that for all $N \in \{ak+1,ak+2,\ldots,(a+1)k\}$, the same vertex of the tree fires the same number of times. Thus, it makes sense to consider the number of fires as a function of $\left\lceil \frac{N}{k} \right\rceil$. Therefore, we introduce a new set of functions:
\[g_i(m,k) = f_i(mk,k).\]

%We wrote a program to find $g(m,k)$ of $m, k \leq 1000$; the full data can be found at this spreadsheet \url{https://docs.google.com/spreadsheets/d/1y-vwPfH192JlxLnWHYgQh1dqYv_ZaaIaujwTwTTaaOo/edit?usp=sharing}.

Table~\ref{tab:g0} shows values of $g_0(m,k)$ for $1 \le m \le 10$ and $2 \le k \le 6$.

\begin{table}[ht!]
\begin{center}
\begin{tabular}{|r|r|r|r|r|r|r|r|r|r|r|}
\hline
$k$\textbackslash $m$ & 1 & 2 & 3 & 4 & 5 & 6 & 7 & 8 & 9 & 10 \\
\hline
2 & 0 & 1 & 2 & 4 & 5 & 7 & 8 & 11 & 12 & 14 \\
\hline
3 & 0 & 1 & 2 & 3 & 5 & 6 & 7 & 9 & 10 & 11 \\
\hline
4 & 0 & 1 & 2 & 3 & 4 & 6 & 7 & 8 & 9 & 11 \\
\hline
5 & 0 & 1 & 2 & 3 & 4 & 5 & 7 & 8 & 9 & 10 \\
\hline
6 & 0 & 1 & 2 & 3 & 4 & 5 & 6 & 8 & 9 & 10 \\
\hline
\end{tabular}
\end{center}
\caption{Values of $g_0(m,k)$ for $1 \le m \le 10$ and $2 \le k \le 6$.}
\label{tab:g0}
\end{table}

The values for $g_0(m,2)$ were calculated in \cite{IKL24}, and they form sequence A376116 in the OEIS \cite{OEIS}.

The sequence $g_0(m,3)$ has been submitted to the OEIS as sequence A378724 and starts as
\[0,\ 1,\ 2,\ 3,\ 5,\ 6,\ 7,\ 9,\ 10,\ 11,\ 13,\ 14,\ 15,\ 18.\]
%0, 1, 2, 3, 5, 6, 7, 9, 10, 11, 13, 14, 15, 18, 19, 20, 22, 23, 24, 26, 27, 28, 31, 32, 33, 35, 36, 37, 39, 40, 41, 44, 45, 46, 48, 49, 50, 52, 53, 54, 58, 59, 60, 62, 63, 64, 66, 67, 68, 71, 72, 73, 75, 76, 77, 79, 80, 81, 84, 85, 86, 88, 89, 90, 92, 93, 94, 98, 99, 100, 102, 103, 104, 106, 107

Consider the difference sequence $d_i(m, k) = g_i(m+1, k) - g_i(m, k)$. In this section, we are interested in the case $i=0$. The following proposition describes the recursion for $d_0(m, k)$.

\begin{proposition}
\label{prop:d0recursion}
The difference sequence $d_0(m, k) = g_0(m+1, k) - g_0(m, k)$ satisfies the recurrence
\[d_0(m, k) =\begin{cases}
	d_0(\frac{m-1}{k}, k) + 1, & \text{ if $m-1$ is a multiple of $k$},\\
	1, & \text{ otherwise}.
	\end{cases}\]
\end{proposition}

\begin{proof}
We know that
\begin{align*}
d_0(m, k) &= g_0(m+1, k) - g_0(m, k) = f_0(km+k, k) - f_0(km, k) \\
&= m + f_0(m) - (m-1) - f_0(m-1) \\
&= 1 + f_0(m) - f_0(m-1)
\end{align*}
by Corollary~\ref{cor:f0recursion}. Now, since $f_0(x)$ depends only on $\left\lceil \frac{x}{k} \right\rceil$, if we have $\left\lceil \frac{m}{k} \right\rceil = \left\lceil \frac{m-1}{k} \right\rceil$, it means that $m-1$ is not a multiple of $k$. This implies that $d_0(m,k) = 1$. If $m-1$ is a multiple of $k$, then we need to calculate $f_0(m) - f_0(m-1)$. We have $f_0(m) = g_0(\frac{m-1}{k}+1)$ and $f_0(m-1) = g_0(\frac{m-1}{k})$. Therefore, $f_0(m) - f_0(m-1) = g_0(\frac{m-1}{k}+1) - g_0(\frac{m-1}{k}) = d_0(\frac{m-1}{k}, k)$, and $d_0(m, k) = d_0(\frac{m-1}{k}, k) + 1$, as claimed.
\end{proof}

Now, we can state the formula.

\begin{theorem}
\label{thm:d0formula}
The difference sequence $d_{0}(m, k) = g_{0}(m + 1, k) - g_{0}(m, k)$ satisfies the formula
\[d_0(m, k) =\begin{cases}
	n, & \text{ if $m = \frac{k^{n} - 1}{k - 1}$ for some integer $n$},\\
	\nu_{k}((k - 1)m + 1) + 1, & \text{ otherwise}.
	\end{cases}\]
\end{theorem}

\begin{proof}
If $m = \frac{k^{n} - 1}{k - 1}$, then we can prove the formula based on the recurrence and using induction on $n$. When $n = 1$, we have $m = 1$ and $d_{0}(m, k) = 1$,. Thus, the base of the induction is established. Assume that $d_{0}(\frac{k^{n} - 1}{k - 1}, k) = n$. By the recurrence from Proposition~\ref{prop:d0recursion}, we have
\begin{align*}
d_{0} \left(\frac{k^{n + 1} - 1}{k - 1}, k \right) = d_{0}\left(\frac{k^{n} - 1}{k - 1},k \right) + 1 = n + 1.
\end{align*}
Now, if $m$ cannot be represented as $\frac{k^{n} - 1}{k - 1}$ for any integer $n$, then we can prove the formula based on the recurrence and using induction. Our base case is when $\nu_{k}(m(k - 1) + 1) = 0$. Therefore, we see that $m - 1$ is not a multiple of $k$. The recurrence gives us $d_{0}(m, k) = 1 = \nu_{k}(m(k - 1) + 1) + 1$. Thus, the base of induction is established.

Assume for all $m \neq \frac{k^{n - 1} - 1}{k - 1}$, such that $\nu_{k} \left(m(k - 1) + 1 \right) = n - 1$ is true, then $d_{0}(m, k) = n$ satisfies the recurrence. Consider some $m \neq \frac{k^{n} - 1}{k - 1}$ such that $\nu_{k} \left(m(k - 1) + 1 \right) = n$. Since $n > 0$, we know that $m - 1$ must be a multiple of $k$. By the recurrence we have $d_{0}(m, k) = d_{0}(\frac{m - 1}{k}) + 1$. Since $m \neq \frac{k^{n} - 1}{k - 1}$, we have $\frac{m - 1}{k} \neq \frac{k^{n - 1}}{k - 1}$. Notice that $\nu_{k}(\frac{(m - 1)(k - 1)}{k} + 1) = \nu_{k}(\frac{m(k - 1) + 1}{k}) = n - 1$. Thus, we have that $d_{0}(m, k) = d_{0}(\frac{m - 1}{k}) + 1 = n + 1$.
\end{proof}

In other words, excluding the first digit, if $[m]_k$ ends with $i$ ones, we have $d_0(m,k) = g_0(m+1,k) - g_0(m,k) = i+1$.

\begin{remark*}
If $N \neq \frac{k^{n} - 1}{k - 1}$, then the number of trailing ones in the base $k$ representation of $N$, without the first digit, is equal to the number of trailing zeros of $N - \frac{k^{n} - 1}{k - 1}$ in base $k$ or $\nu_k(N - \frac{k^{n} - 1}{k - 1})$. This is the same as $\nu_k(N(k - 1) + 1)$.
\end{remark*}

\begin{example}
The difference sequence $d_0(m,2)$ is sequence A091090 in the OEIS \cite{OEIS}, as shown in \cite{IKL24}. The latter sequence is defined in the OEIS as the number of editing steps (delete, insert, or substitute) to transform the binary representation of $n$ into the binary representation of $n + 1$. Sequence A091090 starts from index 1 as
\[1,\ 1,\ 2,\ 1,\ 2,\ 1,\ 3,\ 1,\ 2,\ 1,\ 3,\ 1,\ 2,\ 1,\ 4,\ 1,\ 2,\ 1,\ \ldots.\]

The difference sequence $d_0(m,3)$ has been submitted to the OEIS as sequence A378725 and starts as
\[1,\ 1,\ 1,\ 2,\ 1,\ 1,\ 2,\ 1,\ 1,\ 2,\ 1,\ 1,\ 3,\ 1,\ 1,\ 2,\ 1,\ 1,\ 2,\ 1,\ 1,\ 3,\ 1,\ \ldots.\]
% c[n_] := c[n] = Which[n == 1, 1, Mod[n, 3] != 1, 1, True, c[(n - 1)/3] + 1]; Array[c, 103, 1]
%mine:1, 1, 1, 2, 1, 1, 2, 1, 1, 2, 1, 1, 3, 1, 1, 2, 1, 1, 2, 1, 1, 3, 1, 1, 2, 1, 1, 2, 1, 1, 3, 1, 1, 2, 1, 1, 2, 1, 1, 4, 1, 1, 2, 1, 1, 2, 1, 1, 3, 1, 1, 2, 1, 1, 2, 1, 1, 3, 1, 1, 2, 1, 1, 2, 1, 1, 4, 1, 1, 2, 1, 1, 2, 1, 1, 3, 1, 1, 2, 1, 1, 2, 1, 1, 3, 1, 1, 2, 1, 1, 2, 1, 1, 4, 1, 1, 2, 1, 1, 2, 1, 1, 3
%k=3: 1, 1, 1, 2, 1, 1, 2, 1, 1, 2, 1, 1, 3, 1, 1, 2, 1, 1, 2, 1, 1, 3, 1, 1, 2, 1, 1, 2, 1, 1, 3, 1, 1, 2, 1, 1, 2, 1, 1, 4, 1, 1, 2, 1, 1, 2, 1, 1, 3, 1, 1, 2, 1, 1, 2, 1, 1, 3, 1, 1, 2, 1, 1, 2, 1, 1, 4, 1, 1, 2, 1, 1, 2, 1, 1
%k=4: 1, 1, 1, 1, 2, 1, 1, 1, 2, 1, 1, 1, 2, 1, 1, 1, 2, 1, 1, 1, 3, 1, 1, 1, 2, 1, 1, 1, 2, 1, 1, 1, 2, 1, 1, 1, 3, 1, 1, 1, 2, 1, 1, 1, 2, 1, 1, 1, 2, 1, 1, 1, 3, 1, 1, 1, 2, 1, 1, 1, 2, 1, 1, 1, 2, 1, 1, 1, 3, 1, 1, 1, 2, 1, 1
\end{example}

One can notice that these sequences all have a similar structure due to a similar repetition pattern. Each sequence begins with $k$ ones. Afterward, there are $k - 1$ ones between every two numbers that are not equal to 1. Furthermore, for any $m > 1$, if we remove all numbers less than $m$, the resulting sequence starts with $k$ occurrences of $m$, followed by $k - 1$ occurrences of $m$ between every two numbers that are not equal to $m$.

In other words, we can construct the difference sequence $d_0(m,k)$ in the following way. Start by creating a sequence of all $1$'s. Then, replace every $k$th $1$ with a $2$ beginning with the $(k+1)$th $1$. Next, replace every $k$th $2$ with a $3$, starting with the $(k+1)$th $2$. Continue this process indefinitely, replacing every $k$th occurrence of $x-1$ with $x$, beginning with the $(k+1)$th $x-1$. This fractal structure of sequences $d_0$ follows from Proposition~\ref{prop:d0recursion}.

\section{The total number of fires}
\label{sec:totalfires}

\subsection{The total number of fires}

We can also calculate the total number of fires on an infinite undirected $k$-ary tree with a self-loop at the root. Let $F(N,k)$ denote the total number of fires in a $k$-ary tree when we start with $N$ chips placed at the root. As we have done previously, we sometimes remove the parameters in the function $F$ when $N$ and $k$ are fixed.

\begin{theorem}
The total number of fires is
\[F(N,k)=\frac{1}{(k-1)^2} \sum_{m=1}^{n-1} (mk^{m+1} - (m+1) k^m + 1) c_m(N,K),\]
where $n = \lfloor \log_{k}(N(k-1) + 1) \rfloor$.
\end{theorem}

\begin{proof}
Totaling the number of fires for all the vertices, we have
\[F(N,k) = \sum_{i=0}^{n-2} k^i f_i.\]
Substituting $f_i$ from Theorem~\ref{thm:fires}, we get
\[F(N,k) = \sum_{i=0}^{n-2} k^i \sum_{j=1}^{n-i-1} \frac{k^j - 1}{k-1} c_{i+j}.\]

Rearranging and replacing $i+j$ with $m$, we get
\begin{align*}
F(N,k) &= \sum_{i=0}^{n-2} \sum_{j=1}^{n-i-1} \frac{k^{i+j} - k^i}{k-1} c_{i+j} \\
     &= \frac{1}{k-1}\left(\sum_{i=0}^{n-2} \sum_{j=1}^{n-i-1} k^{i+j} c_{i+j} - \sum_{i=0}^{n-2} \sum_{j=1}^{n-i-1} k^{i} c_{i+j}\right) \\
     &= \frac{1}{k-1}\left(\sum_{m=1}^{n-1} mk^{m} c_{m} - \sum_{m=1}^{n-1} (1 + k + k^2 + \cdots + k^{m-1}) c_{m} \right)\\
     &= \frac{1}{k-1} \left(\sum_{m=1}^{n-1} mk^{m} c_{m} - \sum_{m=1}^{n-1} \frac{k^m - 1}{k-1} c_{m} \right)\\
     &= \frac{1}{(k-1)^2} \sum_{m=1}^{n-1} (mk^{m+1} - (m+1) k^m + 1) c_m.
\end{align*}
This concludes the proof.
\end{proof}

For $k=2$, the sequence of the total number of fires was calculated in \cite{IKL24}. It is sequence A376131 in the OEIS. The sequence starts from index 1 as
\[0,\ 1,\ 2,\ 6,\ 7,\ 11,\ 12,\ 23,\ 24,\ 28,\ 29,\ 40,\ 41,\ 45,\ \ldots.\]
For $k=3$, the sequence of the total number of fires starts from index 1 as
\[0,\ 1,\ 2,\ 3,\ 8,\ 9,\ 10,\ 15,\ 16,\ 17,\ 22,\ 23,\ 24,\ 42,\ 43,\ 44,\ 49,\ 50,\ 51,\ \ldots.\]
This sequence has been submitted to the OEIS as sequence A378726.
%k=3: 0, 1, 2, 3, 8, 9, 10, 15, 16, 17, 22, 23, 24, 42, 43, 44, 49, 50, 51, 56, 57, 58, 76, 77, 78, 83, 84, 85, 90, 91, 92, 110, 111, 112, 117, 118, 119, 124, 125, 126, 184, 185, 186, 191, 192, 193, 198, 199, 200, 218, 219, 220, 225, 226, 227, 232, 233, 234, 252, 253, 254, 259, 260, 261, 266, 267, 268, 326, 327, 328, 333, 334, 335, 340, 341
%k=4: 0, 1, 2, 3, 4, 10, 11, 12, 13, 19, 20, 21, 22, 28, 29, 30, 31, 37, 38, 39, 40, 67, 68, 69, 70, 76, 77, 78, 79, 85, 86, 87, 88, 94, 95, 96, 97, 124, 125, 126, 127, 133, 134, 135, 136, 142, 143, 144, 145, 151, 152, 153, 154, 181, 182, 183, 184, 190, 191, 192, 193, 199, 200, 201, 202, 208, 209, 210, 211, 238, 239, 240, 241, 247, 248
%k=5: 0, 1, 2, 3, 4, 5, 12, 13, 14, 15, 16, 23, 24, 25, 26, 27, 34, 35, 36, 37, 38, 45, 46, 47, 48, 49, 56, 57, 58, 59, 60, 98, 99, 100, 101, 102, 109, 110, 111, 112, 113, 120, 121, 122, 123, 124, 131, 132, 133, 134, 135, 142, 143, 144, 145, 146, 184, 185, 186, 187, 188, 195, 196, 197, 198, 199, 206, 207, 208, 209, 210, 217, 218, 219, 220
We have a recursive formula for function $F$, similar to the formula for $f_0$ found in Corollary~\ref{cor:f0recursion}.

\begin{proposition}
We have 
\[F(N) = f_0(N) + kF\bigg(\bigg\lceil \frac{N}{k} \bigg\rceil -1\bigg).\]
\end{proposition}

\begin{proof}
The total number of fires is equal to the number of root fires plus the number of fires in all of the $k$ subtrees with roots on level 2. Each subtree receives $\left\lceil \frac{N}{k} \right\rceil -1$ chips from the root, so each would fire $F\left(\left\lceil \frac{N}{k} \right\rceil -1\right)$ times.
\end{proof}

\subsection{Special case of $N = \frac {k^n-1}{k-1}$}

If we start with $N = \frac {k^n-1}{k-1}$ chips at the root, we have $c_i = 1$ for $0 \leq i < n$. This case is special, so we want to look at the total number of fires for such $N$ separately.

\begin{proposition}
\label{prop:F}
We have
\[F\left(\frac{k^{j+1}-1}{k-1}\right) = \dfrac{(k(n-1) -n-1)k^{n}+k(n+1)-n+1}{(k-1)^3}.\]
\end{proposition}

\begin{proof}
Using the formula for the total number of fires, we get that there are a total of
\begin{align*}
 &\sum ^{n-1}_{m=1} \dfrac{mk^{m+1}-mk^m-k^m+1}{(k-1)^2}\\
 &= \dfrac 1{(k-1)^2}\sum ^{n-1}_{m=1} (mk^{m+1}-mk^m-k^m+1) \\
 &= \dfrac 1{(k-1)^2} \left( \sum ^{n-1}_{m=1} mk^{m+1}-\sum ^{n-1}_{m=1} mk^m - \sum ^{n-1}_{m=1}k^m +\sum ^{n-1}_{m=1} 1 \right) \\
 &= \dfrac 1{(k-1)^2} \left( (n-1)k^n - 2 \dfrac{k^n-k}{k-1} + n -1 \right) \\
 &= \dfrac {nk^{n+1}-nk^n-k^{n+1}-k^n+k+nk-n+1}{(k-1)^3} \\
 & = \dfrac{(kn-n-k-1)k^{n}+nk+k-n+1}{(k-1)^3} \\
 & = \dfrac{(k(n-1) -n-1)k^{n}+k(n+1)-n+1}{(k-1)^3} 
\end{align*}
fires.
\end{proof}

For $k=2$, we get
\[F(2^n-1) = (2n-2 -n -1)2^{n} + 2n +2 -n +1 = (n - 3)2^{n} + n + 3.\]
This matches Corollary 3.7 in \cite{MN23}, and it is sequence A045618 in the OEIS, with an index shift.

For $k=3$, we get
\[F\left(\frac{3^n-1}{2}\right) = \frac{1}{8}((3n-3 -n -1)3^{n} + 3n +3 -n +1 )= \frac{1}{8}((2n -4)3^{n} + 2n + 4) = \frac{1}{4}((n -2)3^{n} + n + 2),\]
which is sequence A212337 in the OEIS, with an index shift.

As is often the case in this paper, the difference sequences of our sequences are easier to analyze. For example, for $k=2$, our sequence A045618$(n+1)$ is the sum of the first $n$ terms of sequence A000337, which is a much older sequence that is simpler. Similarly, for $k=3$, the sequence A212337$(n+1)$ is the sum of the first $n$ terms of A014915.

Here, we make a general statement.

\begin{proposition}
The term $F\left(\frac{k^{j+1}-1}{k-1}\right)$ is the sum of the first $j$ terms of the sequence $b(n,k)$, which is defined as $b(1,k) = 1$ and $b(n,k) = nk^{n-1} + b(n-1,k)$.
\end{proposition}

\begin{proof}
First, we show that the closed form of $b(n,k)$ is
\[b(n,k)=\frac{k^n((k-1)n-1)+1}{(k-1)^2}.\]
The closed form satisfies $b(1,k) = \frac{k((k-1)-1)+1}{(k-1)^2} = 1$, so it is left to prove that it satisfies the recurrence. Indeed, the expression $nk^{n-1} + b(n-1,k)$ is equal to
\begin{align*}
n k^{n-1} + & \frac{k^{n-1}((k-1)(n-1)-1)+1}{(k-1)^2} = \frac{k^{n-1}((k-1)^2n + (k-1)(n-1)-1)+1}{(k-1)^2} \\
	& = \frac{k^{n-1}((k^2-k)n - k)+1}{(k-1)^2} = \frac{k^{n}((k-1)n - 1)+1}{(k-1)^2} .
\end{align*}
The partial sums of sequence $b(n,k)$ are given by
\[\sum_{j=1}^n\frac{k^j((k-1)j-1)+1}{(k-1)^2} = \frac{(k-1)\sum_{j=1}^n j k^j-\sum_{j=1}^nk^j+n}{(k-1)^2}.\]
Now, we have
\[\sum_{j=1}^n j k^j=\sum_{j=1}^n\frac{k^{n+1}-k^j}{k-1} = \frac{k^{n+1} ((k-1)n-1)+k}{(k-1)^2} \quad \text{ and } \quad \sum_{j=1}^n k^j=\frac{k^{n+1}-k}{k-1}.\]
Substituting this gives
\[\frac{(k-1) \cdot\frac{k^{n+1} ((k-1)n-1)+k}{(k-1)^2}-\frac{k^{n+1}-k}{k-1}+n}{(k-1)^2}=\frac{k^{n+1} ((k-1)n-1)+k - k^{n+1}+ k +n(k-1)}{(k-1)^3}.\]
This simplifies to
\[\frac{k^{n+1} (kn - n -2)+ k(n+2) - n}{(k-1)^3},\]
which is $F(N,k)$ for $N = \frac{k^{n+1}-1}{k-1}$. This is the same as the sequences we study for index $n+1$.
\end{proof}

For $k=4$, we get
\[F\left(\frac{4^n-1}{3}\right)=\frac{1}{27}((3n-5)4^n+3n+5).\]
This sequence has been submitted to the OEIS as sequence A378727. Starting from index 1, the first few terms are
\[0,\ 1,\ 10,\ 67,\ 380,\ 1973,\ 9710,\ 46119,\ 213600,\ 970905,\ 4349650,\ 19262731,\ \ldots.\]
%Table[((3 n - 5) 4^n + 3 n + 5)/27, {n, 30}]
%Total number of fires in a 4-ary tree with a self-loop at the root when we start with (4^n-1)/3 chips at the root. This number of chips is interesting because the stable configuration has 1 chip for every vertex at the top n layers.
%Every other number is divisible by 10.
%1, 10, 67, 380, 1973, 9710, 46119, 213600, 970905, 4349650, 19262731, 84507460, 367855997, 1590728630, 6840133103, 29269406760, 124713124449, 529394487450, 2239745908435, 9447655468300, 39745309211461, 166799986198910, 698474942207927, 2918999758480880, 12176398992520233, 50707195804467810, 210835182555418779, 875366327571865300.
These terms are the partial sums of sequence A014916.

One can notice that every odd-indexed term is divisible by 10. Here, we make a general statement.

\begin{proposition}
We have that $F\left(\frac{k^{2j+1}-1}{k-1}\right)$ is divisible by $2(k+1)$ for any non-negative $j$.
\end{proposition}

\begin{proof}
By Proposition~\ref{prop:F},
\[F\left(\frac{k^{2j+1}-1}{k-1}\right)=\frac{(k(2n-2)-2n)k^{2n-1}+k(2n)-2n+2}{(k-1)^3}.\]
We want to prove that the above is divisible by $2(k+1)$. We have
\begin{align*}
\frac{1}{2(k+1)}F\left(\frac{k^{2j+1}-1}{k-1}\right) & = \frac{(k(2n-2)-2n)k^{2n-1}+2nk-2n+2}{2(k-1)^3(k+1)} \\
&=\frac{(n-1)(k^{2n}-1)-nk(k^{2n-2}-1)}{(k-1)^3(k+1)}.
\end{align*}
Dividing the numerator by $k^2-1 = (k-1)(k+1)$ gives
\[(n-1)(k^{2n-2}+k^{2n-4}+\cdots+1)-nk(k^{2n-4}+k^{2n-6}+\cdots+1).\]
We want to prove that this expression is divisible by $(k-1)^2$. It would be sufficient to consider this expression as a function of $k$ and to prove that this expression and its derivative equals 0 at $k=1$. When $k=1$, this expression is equal to $(n-1)n-n(n-1)=0$. The derivative of this expression is
\[(n-1) \sum_{i=0}^{n-1}2ik^{2i-1}-n \sum_{i=0}^{n-2}(2i+1)k^{2i}.\]
When $k=1$, the derivative is
\[2(n-1)\sum_{i=0}^{n-1}i-n\sum_{i=0}^{n-2}(2i+1)=(n-1)(n-1)n-n(n-1)^2=0,\]
finishing the proof.
\end{proof}

For our next example, we consider $k=5$.

\begin{example}
For $k=5$, we get
\[F\left(\frac{5^n-1}{4}\right)=\frac{1}{64}((4n-6)5^n+4n+6)=\frac{1}{32}((2n-3)5^n+2n+3).\]
This sequence has been submitted to the OEIS as sequence A378728, which starts from index 1 as
\[0,\ 1,\ 12,\ 98,\ 684,\ 4395,\ 26856,\ 158692,\ 915528,\ 5187989,\ 28991700,\ \ldots.\]
%Table[((4 n - 6) 5^n + 4 n + 6)/64, {n, 30}]
%1, 12, 98, 684, 4395, 26856, 158692, 915528, 5187989, 28991700, 160217286, 877380372, 4768371583, 25749206544, 138282775880, 739097595216, 3933906555177, 20861625671388, 110268592834474, 581145286560060, 3054738044738771, 16018748283386232, 83819031715393068.
These terms are the partial sums of sequence A014917.
\end{example}

\subsection{Differences and schizophrenic numbers}

Similar to $f_i(N, k),$ notice that if $N \in \{ak + 1, ak + 2, \ldots, (a+1)k\}$, then the root and the other vertices of the tree fire the same number of times. Thus, it makes sense to consider the number of fires as a function of $\left\lceil \frac{N}{k} \right\rceil.$ Therefore, we introduce a new function: \[G(m, k) = F(mk, k).\]
Table~\ref{tab:G} shows the values of $G(m,k)$ for small values of $m$ and $k$.
%We wrote a program to find the values of $G(m, k)$ for $m, k \leq 200$. The data is in this spreadsheet(\url{https://docs.google.com/spreadsheets/d/1na5pTMVTBOVuWGnvlk-DZsx2YFdx1jx0EsIQQgrCgcc/edit?usp=sharing}{https://docs.google.com/spreadsheets/d/1na5pTMVTBOVuWGnvlk-DZsx2YFdx1jx0EsIQQgrCgcc/edit?usp=sharing}).

\begin{table}[ht!]
\begin{center}
\begin{tabular}{|r|r|r|r|r|r|r|r|r|r|r|}
\hline
$k$\textbackslash $m$ & 1 & 2 & 3 & 4 & 5 & 6 & 7 & 8 & 9 & 10 \\
\hline
2 & 0 & 1 & 2 & 6 & 7 & 11 & 12 & 23 & 24 & 28 \\
\hline
3 & 0 & 1 & 2 & 3 & 8 & 9 & 10 & 15 & 16 & 17 \\
\hline
4 & 0 & 1 & 2 & 3 & 4 & 10 & 11 & 12 & 13 & 19 \\
\hline
5 & 0 & 1 & 2 & 3 & 4 & 5 & 12 & 13 & 14 & 15 \\
\hline
6 & 0 & 1 & 2 & 3 & 4 & 5 & 6 & 14 & 15 & 16 \\
\hline
\end{tabular}
\end{center}
\caption{Values of $G(m,k)$ for $1 \le m \le 10$ and $2 \le k \le 6$.}
\label{tab:G}
\end{table}

The function of differences proved to be useful in calculating the number of fires for particular vertices. Thus, we introduce a difference function here as well:
\[D(m,k) = G(m+1,k)-G(m,k).\]
Table~\ref{tab:GDiff} shows $D(m,k)$ for small values of $m$ and $k$.

\begin{table}[ht!]
\begin{center}
\begin{tabular}{| c | c | c | c | c | c | c | c | c | c | c |} 
 \hline
 $k$\textbackslash $m$ & 1 & 2 & 3 & 4 & 5 & 6 & 7 & 8 & 9 & 10 \\ [0.5ex] 
 \hline
 2 & 1 & 1 & 4 & 1 & 4 & 1 & 11 & 1 & 4 & 1 \\
 \hline
 3 & 1 & 1 & 1 & 5 & 1 & 1 & 5 & 1 & 1 & 5 \\
 \hline
 4 & 1 & 1 & 1 & 1 & 6 & 1 & 1 & 1 & 6 & 1 \\
 \hline
 5 & 1 & 1 & 1 & 1 & 1 & 7 & 1 & 1 & 1 & 1 \\
 \hline
 6 & 1 & 1 & 1 & 1 & 1 & 1 & 8 & 1 & 1 & 1 \\
 \hline
 %7 & 1 & 1 & 1 & 1 & 1 & 1 & 1 & 9 & 1 & 1 \\
 %\hline
 %8 & 1 & 1 & 1 & 1 & 1 & 1 & 1 & 1 & 10 & 1 \\
 %\hline
 %9 & 1 & 1 & 1 & 1 & 1 & 1 & 1 & 1 & 1 & 11 \\
 %\hline
\end{tabular}
\end{center}
\caption{Values of $D(m,k)$ for $1 \le m \le 10$ and $2 \le k \le 6$.}
\label{tab:GDiff}
\end{table}

%More data is available at: \url{https://docs.google.com/spreadsheets/d/1rkseEybL-6-9ubdqwet_uHLfTIyKoUmsyFzOeSzu04k/edit?gid=1991249386#gid=1991249386}.
For $k=2$, this is sequence A376132, calculated in \cite{IKL24}. It starts as
\[1,\ 1,\ 4,\ 1,\ 4,\ 1,\ 11,\ 1,\ 4,\ 1,\ 11,\ 1,\ 4,\ 1,\ 26,\ \ldots.\]
The distinct values of this sequence are Eulerian numbers, which form sequence A000295, where A000295$(n) = 2^n - n - 1$.

For $k=3$, we get the sequence $D(n,3)$, which has been submitted to the OEIS as sequence A378962. It starts from index 1 as
\[1,\ 1,\ 1,\ 5,\ 1,\ 1,\ 5,\ 1,\ 1,\ 5,\ 1,\ 1,\ 18,\ 1,\ 1,\ 5,\ 1,\ 1,\ 5,\ \ldots.\]

The difference sequences for the total number of fires show a similar pattern to the difference sequences for root fires. The data show that, for a given $k$, both sequences have ones in the same places. Moreover, when the difference sequence for the number of root fires has $n$ in some place, the difference sequence for the total number of fires has the same number, which can be viewed as a function of $n$.

The sequences of the unique values are summarized in Table~\ref{tab:a}. For precision, we note that sequence A000295 is shifted relative to our indexing: $A000295(2) = 1$. However, the database has a sequence that matches our indexing: A125128, as $A125128(n) = A000295(n+1)$. Also, excluding index 1, sequence A130103 matches A000295. Sequence A000340 is also shifted, but in the opposite direction: $A000340(0) = 1$.

\begin{table}[ht!]
\begin{center}
\begin{tabular}{| c | c | c | c | c | c | c | c | c |} 
 \hline
 $k$\textbackslash $m$ & 1 & 2 & 3 & 4 & 5 & 6 & 7			& A\# \\ 
\hline
 2 & 1 & 4 & 11 & 26 & 57 & 120 & 247 		& A000295\\
 \hline
 3 & 1 & 5 & 18 & 58 & 179 & 543 & 1636 		& A000340 \\
 \hline
 4 & 1 & 6 & 27 & 112 & 453 & 1818 & 7279 		& A014825\\
 \hline
 5 & 1 & 7 & 38 & 194 & 975 & 4881 & 24412 		& A014827\\
 \hline
 6 & 1 & 8 & 51 & 310 & 1865 & 11196 & 67183 		& A014829\\
 \hline
 7 & 1 & 9 & 66 & 466 & 3267 & 22875 & 160132 	& A014830\\
 \hline
 8 & 1 & 10 & 83 & 668 & 5349 & 42798 & 342391 	& A014831\\
 \hline
 9 & 1 & 11 & 102 & 922 & 8303 & 74733 & 672604 	& A014832\\ 
 \hline
10 & 1 & 12 & 123 & 1234 & 12345 & 123456 & 1234567 	& A014824\\ 
 \hline
\end{tabular}
\end{center}
\caption{Unique values for $2 \le k \le 10$.}
\label{tab:a}
\end{table}

%k=2:, 1, 4, 11, 26, A000295
%k=3: 1, 5, 18, 58, 179, 543, 1636, 4916, 14757, 44281, A000340
%k=5: 1, 7, 38, 194, 975, 4881, 24412, 122068, 610349, 3051755, A014827
%k=6: 1, 8, 51, 310, 1865, 11196, 67183, 403106, 2418645, 14511880, A014829
%k=7: 1, 9, 66, 466, 3267, 22875, 160132, 1120932, 7846533, 54925741, A014830
%k=8: 1, 10, 83, 668, 5349, 42798, 342391, 2739136, 21913097, 175304786, A014831 
%k=9: 1, 11, 102, 922, 8303, 74733, 672604, 6053444, 54481005, 490329055, A014832
%k=10: 1, 12, 123, 1234, 12345, 123456, 1234567, 12345678, 123456789, 1234567900, A014824

The last row shows an interesting pattern: the $i$th number for $i<k$ is a concatenation of the first $j$ digits. The same is true for sequences in other bases. These sequences start with $a(1) = 1$ and follow the recursion $a(n) = ka(n-1)+n$.

The square roots of the odd numbers in sequence A014824 are irrational, but they display some interesting patterns that mimic rational numbers.

\begin{example} Consider $a(11,10) = 12345679011$. Its square root starts as
\[111111.11110505555555539054166665767340972160955659283519805.\]
There are three blocks of repeated digits here. The number starts with 10 ones, and then shortly after, there is a block of 8 fives and a block of 4 sixes. Compare this with the square root of $a(19,10) = 1234567901234567899$, which starts as 
\[1111111111.11111111010555555555555555510054166666666666625487909722222222175.\]
There are longer blocks of the same digits and a new block of twos. The blocks with repeated digits decrease in size, and the number of digits between blocks increases.
\end{example}

A \textit{schizophrenic number} or \textit{mock rational number} is an irrational number that displays certain characteristics of rational numbers \cite{Toth20}, as shown in the example above. Such numbers can be similarly defined in other bases. They all begin with blocks of ones in their corresponding bases and seem to appear rational, but then the pattern breaks, and they appear irrational. Later on, they seem to appear rational again with another digit until the pattern breaks again. The blocks of rationality shrink, and the blocks of irrationality lengthen until, eventually, the whole thing disintegrates into the full chaos of irrational numbers. The schizophrenic numbers in base $k$ are square roots of odd numbers in the corresponding sequence.

Peter Bala made an interesting comment on sequence A014824 in the OEIS \cite{OEIS}, suggesting that the inverse of a schizophrenic number also exhibits schizophrenic patterns. These patterns are even more striking as blocks of repeated digits always contain the same digit: zero. Presumably, the same pattern holds for other bases.

\begin{example}
Consider $a(11,10) = 12345679011$ and $a(19,10) = 1234567901234567899$ as in the previous example. The first few digits of $\sqrt{\frac{1}{a(11,10)}}$ and $\sqrt{\frac{1}{a(19,10)}}$ are
\[0.0000090000000004905000000400983750036422690628473814118700156165\]
and
\[0.0000000009000000000000000008145000000000000011056837500000000016.\]
There are blocks of repeated zeros. Ignoring the initial zeros, the block sizes decrease, and the number of digits between blocks increases. The pattern is similar to the schizophrenic pattern, except that the blocks always contain zeros.
\end{example}

We define functions $a(n,k)$ for all $k \geq 2$ as $a(n,k) = ka(n - 1,k) + n$ for $n \geq 2$ and $a(1,k) = 1$. This is the same recurrence mentioned in the sequences in Table~\ref{tab:a}. This recursion explains why the first $k$ terms of each sequence written in base $k$ are the concatenations of the first $k$ digits.

We first get the closed formula for $a(n,k)$.

\begin{lemma}
\label{lem:ank}
We have that
\[a(n,k) = \frac{k^{n + 1} - (k - 1)n - k}{(k - 1)^2}.\]
\end{lemma}

\begin{proof}
We prove this with induction on $n$. The base case is $n = 1$, and the formula gives $a(1,k) = \frac{k^{2} - 2k + 1}{(k - 1)^2} = 1$, which is true by definition.

Assume $a(n,k) = \frac{k^{n + 1} - (k - 1)n - k}{(k - 1)^2}$. The recursion tells us that $a(n + 1,k) = a(n,k) + n + 1$. Thus,
\begin{align*}
a(n + 1,k) &= \frac{k^{n + 2} - (k^2 - k)n - k^2 + (k - 1)^2n + (k - 1)^2}{(k - 1)^2} \\
&= \frac{k^{n + 2} - (k - 1)(n + 1) - k}{(k - 1)^2},
\end{align*}
concluding the proof.
\end{proof}

We are now ready to prove the observation about the differences mentioned earlier.

\subsection{Formulae for differences}

In Proposition~\ref{prop:d0recursion} and Theorem~\ref{thm:d0formula}, we provided the recurrence and the formula for the function $d_0$. In this section, we do the same for function $D$.

\begin{proposition}
The difference sequence $D(m, k) = G(m+1, k) - G(m, k)$ satisfies the recurrence
\[D(m, k) =\begin{cases}
	d_0(m) + kD(\frac{m-1}{k}), & \text{ if $m-1$ is a multiple of $k$},\\
	1, & \text{ otherwise}.
	\end{cases}\]
\end{proposition}

\begin{proof}
Assume that the second parameter of all functions is $k$. We have
\begin{align*}
  D(m) &= G(m+1) - G(m) = F(km+k) - F(km) \\
  &= f_0(km+k) - f_0(km) + k(F(m) - F(m-1)) \\
  &= d_0(m) + k(F(m) - F(m-1)).
\end{align*}
Note that $F(m) - F(m-1) \neq 0$ if and only if $k$ divides $m-1$. In this case, we have
\begin{align*}
  D(m) &= d_0(m) + k(F(m) - F(m-1)) \\
  &= d_0(m) + k\left(G\left(\frac{m-1}{k} + 1\right) - G\left(\frac{m-1}{k}\right)\right) \\
  &= d_0(m) + kD \left ( \frac{m - 1}{k} \right).
\end{align*}
If $k$ does not divide $m-1$, then $F(m) - F(m-1) = 0$ and $d_0(m) = 1$. Therefore, 
\[D(m) = d_0(m) + k(F(m) - F(m-1)) = 1,\]
concluding the proof.
\end{proof}

Using the above proposition, we can prove the following theorem.

\begin{theorem}
We have
\[D(m, k) = G(m+1, k) - G(m, k) = a(d_0(m, k), k).\]
\end{theorem}

\begin{proof}
We prove this using induction on the value of $d_0(m, k)$. If $d_0(m) = 1$, then $m-1$ is not a multiple of $k$. Thus, $D(m,k) = 1 = a(1, k)$, as expected. We assume that we have proven the statement for all $d_0(m) = d$ for some value $d$, and we need to prove the statement for $d_0(m) = d+1$. This means that $m-1$ is a multiple of $k$, so we can substitute:

\[D(m) = d_0(m) + kD\left(\frac{m-1}{k}\right).\]

Note that $d_0(\frac{m-1}{k}) = d_0(m) - 1 = d$. Therefore, by induction, we know that

\[D\left(\frac{m-1}{k}\right) = a\left(d_0\left(\frac{m-1}{k}, k\right), k\right) = a(d, k).\]
Thus,

\[D(m) = d+1 + ka(d, k) = a(d+1, k) = a(d_0(m, K), k).\]
\end{proof}

Now, we can use Lemma~\ref{lem:ank} to produce an explicit formula.

\begin{corollary}
We have 
\[D(m, k) = G(m+1, k) - G(m, k) = \frac{k^{j + 1} - (k - 1)j - k}{(k - 1)^2},\]
where $j = \nu_k(m - \frac{k^{n}-1}{k-1}) + 1$, if $m \neq \frac{k^{n} - 1}{k - 1}$. Otherwise, $j = n$.
\end{corollary}

\section{Acknowledgments} 

This project was completed as part of the MIT PRIMES STEP program, and we are grateful to the program and its staff.

\end{document}